\documentclass[a4aper,pagesize]{scrartcl}
\usepackage[utf8]{inputenc}
\usepackage[T1]{fontenc}
\usepackage[english]{babel}
\usepackage{enumitem}
\usepackage{amsmath}
\usepackage{amssymb}
\usepackage{amsthm}
\usepackage{mathtools}
\usepackage{graphicx}
\usepackage{pdfpages}
\usepackage{float}
\usepackage{subfigure}
\usepackage{stmaryrd}
\usepackage{tikz}
\usetikzlibrary{decorations.pathreplacing}
\usepackage{wrapfig}
\usepackage{blindtext}
\usepackage{xcolor,colortbl}
\usepackage{todonotes}
\usepackage{colonequals}

\usepackage[section]{placeins}
\usepackage[backend = biber]{biblatex}
\newtheorem{lemma}{Lemma}[section]
\newtheorem{theorem}[lemma]{Theorem}
\newtheorem*{theorem*}{Theorem}
\newtheorem{corollary}[lemma]{Corollary}
\newtheorem{proposition}[lemma]{Proposition}

\theoremstyle{definition}
\newtheorem{definition}[lemma]{Definition}

\theoremstyle{remark}
\newtheorem{remark}[lemma]{Remark}

\newcommand*\circled[1]{\tikz[baseline=(char.base)]{
\node[shape=circle,draw, minimum size = 13pt,inner sep = -2pt ] (char) {\strut #1};}}

\newenvironment{rAnd}{\begin{rcases}}{\end{rcases}\hspace{-3pt}\circled{$\wedge$}}
\newenvironment{lAnd}{\circled{$\wedge$}\hspace{-3pt}\begin{cases}}{\end{cases}}
\newenvironment{rOr}{\begin{rcases}}{\end{rcases}\hspace{-3pt}\circled{$\vee$}}
\newenvironment{lOr}{\circled{$\vee$}\hspace{-3pt}\begin{cases}}{\end{cases}}
\newenvironment{rXor}{\begin{rcases}}{\end{rcases}\hspace{-3pt}\circled{$\underline{\vee}$}}
\newenvironment{lXor}{\circled{$\underline{\vee}$}\hspace{-3pt}\begin{cases}}{\end{cases}}
\newenvironment{rEquiv}{\begin{rcases}}{\end{rcases}\hspace{-3pt}\circled{$\Leftrightarrow$}}
\newenvironment{lEquiv}{\circled{$\Leftrightarrow$}\hspace{-3pt}\begin{cases}}{\end{cases}}

\newcommand{\LAnd}[1]{\begin{lAnd}#1\end{lAnd}}

\DeclareTextFontCommand{\emph}{\bfseries}

\title{Most Iterations of Projections Converge}
\date{\today}
\author{Daylen K. Thimm\footnote{daylen.thimm@student.uibk.ac.at\\The author was supported by the doctoral scholarship of the University of Innsbruck.}}

\begin{document}

\maketitle
\begin{center}
  \large
  Institut für Mathematik, Universität Innsbruck\\
  Technikerstraße 13, 6020 Innsbruck, Austria
\end{center}
\section*{Abstract}
Consider three closed linear subspaces $C_1, C_2,$ and $C_3$ of a Hilbert space $H$ and the orthogonal projections $P_1, P_2$ and $P_3$ onto them. Halperin showed that a point in $C_1\cap C_2 \cap C_3$ can be found by iteratively projecting any point $x_0 \in H$ onto all the sets in a periodic fashion. The limit point is then the projection of $x_0$ onto $C_1\cap C_2 \cap C_3$. Nevertheless, a non-periodic projection order may lead to a non-convergent projection series, as shown by Kopeck\'{a}, M\"{u}ller, and Paszkiewicz. This raises the question how many projection orders in $\{1,2,3\}^\mathbb{N}$ are ``well behaved'' in the sense that they lead to a convergent projection series. Melo, da Cruz Neto, and de Brito provided a necessary and sufficient condition under which the projection series converges and showed that the ``well behaved'' projection orders form a large subset in the sense of having full product measure. We show that also from a topological viewpoint the set of ``well behaved'' projection orders is a large subset: it contains a dense $G_\delta$ subset with respect to the product topology. Furthermore, we analyze why the proof from the measure theoretic case cannot be directly adapted to the topological setting.\\

\section{Introduction}\label{sec:intro}
Let $C_1$ and $C_2$ be two closed and convex subsets of a Hilbert space with nonempty intersection. We are interested in the problem of finding some point in the intersection $C_1\cap C_2$, which is usually is referred to as the \emph{convex feasibility problem}. Von Neumann proposed the method of alternating projections for solving it \cite{vonNeumann1949rings}. As a more general setting, we may consider a Hilbert space $(X,\langle \cdot,\cdot\rangle)$, an index set $I\coloneqq\{1,\dots,N\}$ for $N\in\mathbb{N}$, and a finite number of closed convex sets $(C_n)_{n\in I}$ within $X$. By $P_n$ we denote the distance minimizing projection from $X$ to the set $C_n$, $n\in I $, which is well defined, as $C_n$ is convex and closed. Let us choose some starting point $\xi_0\in X$ and a sequence $x = (x_n)_{n\in\mathbb{N}}\in I^{\mathbb{N}}$. We denote the space of all such sequences by $K\coloneqq I^{\mathbb{N}}$. Iteratively we may now define a projection sequence $(\xi_n)_{n\in\mathbb{N}}$ by
\[ \xi_n = P_{x_n}(\xi_{n-1}), \quad n\in\mathbb{N}.\]
We call this construction the method of alternating projections. The hoped for result of this method is that the sequence $(\xi_n)_{n\in\mathbb{N}}$ converges to a point in the intersection \mbox{$C_1\cap\dots\cap C_N$}. Following this idea, von Neumann proved in \cite{vonNeumann1949rings} that at least there is convergence when $X$ is a Hilbert space, $N=2$, and $C_1$ and $C_2$ are closed linear subspaces. A simple geometric proof of von Neumann's theorem is provided by Kopeck\'{a} and Reich in \cite{kopecka2004note}. In \cite{halperin1962product} Halperin showed that for periodic $x$ von Neumann's result holds for any finite number of sets. Sakai was in \cite{sakai1995strong} able to extend this to quasi-periodic sequences.
\begin{definition}
  We call a sequence $x=(x_n)_{n\in\mathbb{N}} \in K$ quasi-periodic if
  \[\exists m \in \mathbb{N}\colon \forall k \in \mathbb{N} \colon \{x_k, x_{k+1},\dots, x_{k+m-1}\} = I.\]
  The smallest such $m$ we call the quasi-period of $x$.
\end{definition}
This condition imposes a uniform bound on the distance of occurrences of indices in $x$, similar to periodicity.\\

Kopeck\'{a}, Müller, and Paszkiewicz showed in \cite{kopecka2014product, kopecka2017strange} that not every $x\in K$ induces a converging sequence $(\xi_n)_{n\in\mathbb{N}}$. They constructed a counterexample in an infinite-dimensional Hilbert space of three closed linear subspaces such that for any $0\neq \xi_0 \in X$ there is a sequence $x\in K$ where the method of alternating projections does not strongly converge. This motivates the following question:
\begin{quote}
  How large is the set of sequences $x\in K$ for which $(\xi_n)_{n\in\mathbb{N}}$ is strongly convergent?
\end{quote}
Melo, da Cruz Neto, and de Brito studied this question within the context of measure. They studied $K$ as a probability space $(K, \mathbb{P})$ by using the product measure of $\mathbb{P}_I$ over the index set $\mathbb{N}$, where $\mathbb{P}_I (\{1\})=\dots =\mathbb{P}_I (\{N\})=\frac{1}{N}$. In Proposition 4.3 of \cite{BCM2022} they showed that the up until now considered sets of periodic and quasi-periodic sequences are null sets with respect to $\mathbb{P}$. Furthermore, they introduced a more general notion than quasi-periodic sequences: the notion of quasi-normal sequences.
\begin{definition}[Definition 4.2 in \cite{BCM2022}]\label{def:quasi_normal}
  We call a sequence $(x_n)_{n\in\mathbb{N}}\in K$ quasi-normal if there exists an $L\in\mathbb{N}$ and a sequence of disjoint blocks $(\mathcal{R}_k)_{k\in\mathbb{N}}$ of consecutive elements of $(x_n)_{n\in\mathbb{N}}$ with $L$ terms, where each block $\mathcal{R}_k$ contains every element of $I$ so that there exists a function $f\colon \mathbb{N} \to (0,\infty)$ with $\lim_{r_k\to\infty} f(r_k) = \infty$ such that
  \[\sum_{k\in\mathbb{N}} \frac{1}{r_k \cdot f(r_k)} = \infty,\]
  where $x_{r_k}$ is the first element of the block $\mathcal{R}_k$ and $(r_k)_{k\in\mathbb{N}}$ is an increasing sequence. We denote the set of quasi-normal sequences by $\mathcal{N}$.
\end{definition}
This type of sequences is interesting, as they guarantee strong convergence of the projection sequence $(\xi_n)_{n\in\mathbb{N}}$ if the sequence $(\xi_n)_{n\in\mathbb{N}}$ has at least one accumulation point; see Theorem 4.1 of \cite{BCM2022}. In fact, the statement there even works for Hadamard spaces. Additionally, Proposition 4.2 of \cite{BCM2022} states that quasi-normal sequences form a full measure subset of $K$.\\

In this paper, we consider topological and metric notions of large and small subsets instead of measure theoretic versions. More precisely, we consider the question whether the set of sequences $x\in K$ leading to a strongly convergent projection sequence $(\xi_n)_{n\in\mathbb{N}}$ is a large subset in a topological and metric sense and, in particular, if the quasi-normal sequences already form such a large subset. Specifically, we are interested in the topological notions of meager and dense $G_\delta$ subsets and the stronger notions of $\sigma$-porous and co-$\sigma$-porous subsets as a metric notion of small and large subsets, respectively; see e.g. \cite{Zajicek}.

\begin{definition}[($\sigma$-)porous subset]
  A subset $A$ of a metric space $(X,d)$ is called \emph{porous at $x\in A$} if there are $\varepsilon_0>0$ and $\alpha>0$ such that for every $\varepsilon\in (0,\varepsilon_0)$ there is a point $y\in X\setminus A$ with $\rho(x,y)<\varepsilon$ and $B(y,\alpha \varepsilon)\cap A=\emptyset$ or, put differently,
  \[\textnormal{ $A \textnormal{ porous at } x:\Leftrightarrow$ \resizebox{11.4cm}{!}{$\exists \varepsilon_0>0:\exists \alpha>0: \forall \varepsilon \in (0,\varepsilon_0): \exists y\in X\setminus A: \rho(x,y) < \varepsilon \wedge B(y,\alpha \varepsilon)\cap A = \emptyset.$}}\]
  The set $A$ is called \emph{porous} if it is porous at all its points. A subset of $X$ is called \emph{$\boldsymbol\sigma$-porous} if it is a countable union of porous sets.	We call a set \emph{co-porous} or \emph{co-$\boldsymbol\sigma$-porous} if its complement is porous or $\sigma$-porous, respectively.
\end{definition}

For $N \in \mathbb{N}$ we equip $I = \{1,\dots,N\}$ with the discrete topology and $K \coloneqq I^\mathbb{N}$ with the product topology. It is well known, that the topology on $K$ is induced by the complete metric
\begin{align*}
  d(x, y) \coloneqq \max \{2^{-j} d_0(x_j, y_j) \colon j\in \mathbb{N}\},
\end{align*}
where $d_0$ denotes the discrete metric on $I$. Note that for $x\in K$ and $j\in\mathbb{N}$ we have
\begin{equation} \label{smallDistanceImpliesSameEntriesAtBeginning}
B(x, 2^{-j}) = \{y\in K \colon y_1 = x_1 \wedge \dots \wedge y_j = x_j\}.
\end{equation}
Using these topological notions of large and small subsets, one would hope for an analogous result to that of Melo, da Cruz Neto, and de Brito stating that almost all projection orders induce a strongly convergent projection series, or that the quasi-normal sequences $\mathcal{N}$ form a large subset. In the measure theoretic case this was shown by defining a stronger condition resulting in a subset $\mathcal{N}_0$ of the quasi-normal sequences and showing that already this set has full measure. The subset $\mathcal{N}_0$ is defined as the set of all $x\in K$ satisfying the assumption of the following proposition.

\begin{proposition}[Proposition 4.1 in \cite{BCM2022}]\label{prop:quasi_normal}
  Let $(x_n)_{n\in\mathbb{N}}\in K$. Suppose that there exists an $L\in\mathbb{N}$, a sequence of disjoint blocks $(\mathcal{R}_k)_{k\in\mathbb{N}}$ of consecutive elements of $(x_n)_{n\in\mathbb{N}}$ with $L$ terms, where each block $\mathcal{R}_k$ has every element of $I$. For each $k\in\mathbb{N}$ let $\mathcal{S}_k$ be the block formed by the elements between $\mathcal{R}_{k-1}$ and $\mathcal{R}_k$, which may eventually be empty. Thus, the sequence $(x_n)_{n\in\mathbb{N}}$ can be seen as follows:
  \[\mathcal{S}_1 \mathcal{R}_1\mathcal{S}_2 \mathcal{R}_2\dots\mathcal{R}_{k-1}\mathcal{S}_k \mathcal{R}_k\dots\]
  Let $|\mathcal{S}_k|$ be the number of elements of this block and $c>0$ a constant. If for all $k \in \mathbb{N}$ we have
  \[\sum_{i = 1}^k |\mathcal{S}_i|\leq c k,\]
  then the sequence $(x_n)_{n\in\mathbb{N}}$ is quasi-normal.
\end{proposition}

  However, from a topological point of view, as the author showed in \cite{thimm2023OnAMeager}, the equivalent statement to the measure theoretic result of $\mathcal{N}_0$ being large does not hold, in fact, the set $\mathcal{N}_0$ turns out to be meager. Apparently, the strategy of Melo, da Cruz Neto, and de Brito cannot be transferred to the topological setting.\\

The current work instead shows directly that the set of quasi-normal sequences is a large subset and analyzes exactly which aspects of $\mathcal{N}_0$ make it a small subset. Proving co-$\sigma$-porosity of $\mathcal{N}$ could however not be achieved. Instead we were able to prove a similar result using a weaker version of porosity, namely the notion of $\phi$-lower porosity where the size of the ``holes'' is still controlled, but it no longer needs to be linear with respect to the distance. This weaker notion of porosity can be found for example in Definition~2.1 in~\cite{Zajicek}.\\

\begin{definition}[($\sigma$-)$\phi$-porous set]
Let $(X, \rho)$ be a metric space, $A\subseteq X$, $x\in A$, $\eta>0$ and $\phi\colon (0,\eta)\to (0,\infty)$ be an increasing function. Given $r>0$, we define
\[
  \gamma(x,r,A) \coloneqq \sup \{s>0 \colon  \exists y \in X \colon B(y,s)\subseteq B(x,r)\setminus A\},
\]
where we use the convention that $\sup\emptyset = - \infty$. The set $A$ is called \emph{$\phi$-lower porous at $x$} if
\[
  \liminf_{r\to 0} \frac{\phi(\gamma(x,r,A))}{r} > 0.
\]
A set $A$ is called $\phi$-lower porous if it is $\phi$-lower porous at all points $x\in A$. Since we do not use the notion of $\phi$-upper porosity, we refer to $\phi$-lower porous sets simply as \mbox{\emph{$\phi$-porous sets}}. Again, we call a set $\sigma$-$\phi$-porous, if it is a countable union of $\phi$-porous sets. Also, we call a set co-$\psi$-porous or co-$\sigma$-$\phi$-porous if it is the complement of a $\phi$-porous or $\sigma$-$\phi$-porous set, respectively.
\end{definition}
We will heavily use the following characterization of $\phi$-porous sets due to M.~Dymond.
\begin{lemma}[Lemma~2.2 in~\cite{Dymond2022}]\label{lem:PorousMichael}
  Let $(X, \rho)$ be a metric space without isolated points, $A\subseteq X$, $x\in A$, $\eta>0$ and
  \[
    \phi\colon (0,\eta) \to (0,\infty)
  \]
  be a strictly increasing, concave function with $\lim_{t\to 0} \phi(t)=0$. Then $A$ is $\phi$-porous at $x$ if and only if there are an $\varepsilon_0>0$ and an $\alpha\in(0,1)$ such that for every $\varepsilon\in (0,\varepsilon_0)$ there is a $y\in X\setminus A$ such that $0<\rho(x,y)\leq \varepsilon$ and $B(y,\phi^{-1}(\alpha \varepsilon)) \cap A = \emptyset$, or put differently
  \[\textnormal{ $A \, \phi\textnormal{-porous at } x:\Leftrightarrow$ \resizebox{11.2cm}{!}{$\exists \varepsilon_0>0:\exists \alpha \in (0,1): \forall \varepsilon \in (0,\varepsilon_0): \exists y\in X \setminus A: \LAnd{\rho(x,y) \leq \varepsilon \\ B(y, \phi^{-1}(\alpha\varepsilon))\cap A = \emptyset}$\hspace{-4mm}.}}\]
\end{lemma}

This relatively new characterization not only proved to be useful in \cite{Dymond2022} but also in \cite{bargetz2023generic}.\\

Using this characterization we were able to prove our main result, Theorem \ref{thm:quasi_normals_co_sigma_phi_porous}, which states that the set of quasi-normal sequences is contained in a co-$\sigma$-$\phi$-porous subset of $K$ with respect to a metric generating the product topology on $K$.\\

Furthermore, in Section \ref{sec:ext} we are able to exactly pinpoint why $\mathcal{N}_0$ is a meager subset and formulate similar, more general subsets that are residual. Lastly, in Section \ref{sec:finite_index_occurrence} we show that almost all sequences $x\in K$ contain all indices in $I$ infinitely often.

\section{Main Result}
In order to prove our main result, Theorem \ref{thm:quasi_normals_co_sigma_phi_porous}, we will first simplify the rather complicated Definition \ref{def:quasi_normal} of the quasi-normal sequences $\mathcal{N}$. Following some preparation, we are able to state Lemma \ref{lem:quasiNormalFullSimplified}, which gives us a characterization of quasi-normal sequences that eliminates the use of the function $f$ of the quasi-normal definition and also uses a constructive way of defining a partition instead of merely demanding the existence of a certain partition.

\subsection{Simplifying the Quasi-Normal Definition}
\begin{proposition}\label{prop:quasiNormalSimplified}
  A sequence $(x_n)_{n\in\mathbb{N}}\in K$ is quasi-normal if and only if there exists an $L\in\mathbb{N}$ and a sequence of disjoint blocks $(\mathcal{R}_k)_{k\in\mathbb{N}}$ of consecutive elements of $(x_n)_{n\in\mathbb{N}}$ with $L$ terms, where each block $\mathcal{R}_k$ contains every element of $I$  such that
  \[\sum_{k\in\mathbb{N}} \frac{1}{r_k} = \infty,\]
  where $x_{r_k}$ is the first element of the block $\mathcal{R}_k$ and $(r_k)_{k\in\mathbb{N}}$ is an increasing sequence.
\end{proposition}
\begin{proof}
  We first show that every quasi-normal sequence satisfies the conditions of the proposition. Let $x$ be quasi-normal and let $(r_k)_{k\in\mathbb{N}}$ and $f$ be the respective quantities from the definition of quasi-normality. In particular, we have that $\lim_{t\to\infty}f(t) = \infty$ and
  \[\sum_{k\in\mathbb{N}}\frac{1}{r_k \cdot f(r_k)} = \infty.\]
  Now, since $f$ tends towards infinity, we may define the smallest index $k_f$ after which all $f(r_j)$ for $j \geq k_f$ are greater or equal to $1$,
  \[k_f = \min\{k\in\mathbb{N} \colon \forall j\geq k \colon f(r_j) \geq 1\}.\]
  Then we have
  \begin{align*}
    \infty = \sum_{k\in\mathbb{N}} \frac{1}{r_k f(r_k)}
           = \sum_{k = 1}^{k_f-1} \frac{1}{r_k f(r_k)} + \sum_{k = k_f}^{\infty} \frac{1}{r_k f(r_k)}
           \leq \sum_{k = 1}^{k_f-1} \frac{1}{r_k f(r_k)} + \sum_{k = k_f}^{\infty} \frac{1}{r_k},
  \end{align*}
  and in particular that $\sum_{k\in\mathbb{N}}\frac{1}{r_k} = \infty$.\\

  For the other direction let $x\in K$ and $(\mathcal{R}_k)_{k\in \mathbb{N}}$ be a partition with starting indices $(r_k)_{k\in\mathbb{N}}$ satisfying the condition of the proposition. Since $\sum_{k\in \mathbb{N}}\frac{1}{r_k} = \infty$, we can choose a sequence $(k_l)_{l\in\mathbb{N}}$ of positive integers such that $k_1=1$ and
  \[\forall l\in\mathbb{N} \colon \sum_{k = k_l}^{k_{l+1}-1} \frac{1}{r_k} \geq l.\]
  Note that this is equivalent to
  \[\forall l\in\mathbb{N} \colon \sum_{k = k_l}^{k_{l+1}-1} \frac{1}{r_k \cdot l} \geq 1.\]
  Now, define $f$ by $f(r_k) = l$ for $k_l \leq k < k_{l+1}$ and observe that it is divergent. We then have
  \[\sum_{k\in\mathbb{N}} \frac{1}{r_k f(r_k)} = \sum_{l\in\mathbb{N}} \sum_{k = k_l}^{k_{l+1}-1} \frac{1}{r_k f(r_k)} \geq \sum_{l\in \mathbb{N}} 1 = \infty.\]
  Hence, $x$ is quasi-normal.
\end{proof}
Now that we have simplified the definition of the quasi-normal sequences, we are able to define a unique partition type that will maximize the sum in the previous proposition.
\begin{definition}[Greedy $L$-partition]
  Let $L\in\mathbb{N}$ with $L\geq |I|$ be given. Furthermore, let $x\in K$. Let
  \[r_1 \coloneqq \min\{r \in\mathbb{N}\colon \{x_r, x_{r+1},\dots,x_{r+L-1}\} = I\},\]
  and for $k\in \mathbb{N}$, $k\geq 2$ let
  \[r_k \coloneqq \min\{r \in\mathbb{N}\colon r > r_{k-1}+L-1 \wedge \{x_r, x_{r+1},\dots,x_{r+L-1}\} = I\}\]
  provided that these minima exist. This results in the blocks
  \[\mathcal{R}_k = (x_{r_k},\dots,x_{r_{k+L-1}}),\quad k\in\mathbb{N}.\]
  We call this partition of $x$ the greedy $L$-partition. Note that by definition every block $\mathcal{R}_k$ contains all elements of $I$. If any of the above minima does not exist, we say that the greedy $L$-partition of $x$ does not exist. In this case, let $\tilde{k}$ be the largest $k\in\mathbb{N}$ for which $r_k$ is well defined. We then say that the greedy $L$-partition of $x$ exists up until at most block $\tilde{k}$. Since for given $x$ and $L$ the partition is fully determined by the starting indices $(r_k)_{k\in\mathbb{N}}$ of $(\mathcal{R}_k)_{k\in\mathbb{N}}$, we will interchangeably instead of the partition itself refer to $(r_k)_{k\in\mathbb{N}}$ as the greedy $L$-partition of $x$.
\end{definition}
The next proposition shows in what sense the greedy $L$-partition is greedy.
\begin{proposition}\label{prop:greedy_L_partition_is_greedy}
  Let $L\in\mathbb{N}$ with $L\geq N$ be given. Furthermore, let $x\in K$ and $(r_k)_{k\in\mathbb{N}}$ be the indices of the greedy $L$-partition of $x$, and $(\tilde{r}_k)_{k\in\mathbb{N}}$ a partition of $x$ different from the former with disjoint blocks of length $L$ in which all elements of $I$ occur. Then we have
  \[\sum_{k\in\mathbb{N}} \frac{1}{r_k} \geq \sum_{k\in\mathbb{N}}\frac{1}{\tilde{r}_k}.\]
  If the right hand side converges we even get a strict inequality.
\end{proposition}
\begin{proof}
  As all $r_k$ are chosen minimally, we have that $r_k\leq \tilde{r}_k$. To further elaborate on this we use induction: Consider $r_1$ and $\tilde{r}_1$. Here it is clear that $r_1 \leq \tilde{r_1}$, since $r_1$ marks the beginning of the first possible block of length $L$ containing all elements of $I$. Now assume that for $n\leq m\in\mathbb{N}$ we have $r_n\leq\tilde{r}_n$. Then $r_{m}$ marks the beginning of the $m$-th block of the greedy $L$-partition. This block extends up to index $r_{m+L-1}$. Now since $r_m\leq\tilde{r}_m$ we have that the $m$-th block in the partition $(\tilde{r}_k)_{k\in\mathbb{N}}$ extends up to index $\tilde{r}_m+L-1$ and thus the same or farther than the $m$-th block of $(r_k)_{k\in\mathbb{M}}$. Since in the greedy $L$-partition $(r_k)_{k\in\mathbb{N}}$ the $(m+1)$-th block is chosen to begin after index $r_{m}+L-1$ with minimal distance, it may already start in a part that the $m$-th block of $(\tilde{r}_k)_{k\in\mathbb{N}}$ still occupies. If this is the case, then clearly $r_{m+1} \leq \tilde{r}_{m+1}$. In the case that the $(m+1)$-th block of $(r_k)_{k\in\mathbb{N}}$ begins after the $m$-th block of $(\tilde{r}_k)_{k\in\mathbb{N}}$, then by the greedy nature of $(r_k)_{k\in\mathbb{N}}$ we again have that $r_{m+1}\leq\tilde{r}_{m+1}$. It follows that indeed $r_k\leq \tilde{r}_k$ for all $k\in\mathbb{N}$.\\

  Since $(r_k)_{k\in\mathbb{N}}\neq (\tilde{r}_k)_{k\in \mathbb{N}}$, we know that for at least one $k\in\mathbb{N}$ we have that $r_k < \tilde{r}_k$. This directly implies the result.
\end{proof}
The previous proposition again allows us to reformulate the definition of quasi-normal sequences.
\begin{lemma}\label{lem:quasiNormalFullSimplified}
  A sequence $x=(x_n)_{n\in\mathbb{N}}\in K$ is quasi-normal if and only if there exists an $L\in\mathbb{N}$ such that the greedy $L$-partition $(r_k)_{k\in\mathbb{N}}$ of $x$ exists and
  \[\sum_{k\in\mathbb{N}} \frac{1}{r_k} = \infty.\]
\end{lemma}

\subsection{Proving the Main Theorem}
We are going to show that the set of quasi-normal sequences $\mathcal{N}$ contains a co-$\sigma$-$\phi$-porous subset, that is, that the complement $K\setminus\mathcal{N}$ is contained in a $\sigma$-$\phi$-porous subset. For this let us for $L\in\mathbb{N}$ with $L\geq N$ and $M\in\mathbb{N}$ define the set
\[\mathcal{A}_{L,M}\coloneqq \left\{x\in K\colon \textnormal{greedy $L$-partition $(r_k)_{k\in\mathbb{N}}$ of $x$ exists and} \sum_{k\in\mathbb{N}} \frac{1}{r_k} < \frac{1}{L}\log(M) \right\}\]
and for $k\in\mathbb{N}$ the set
\[\mathcal{B}_{L,k} \coloneqq \{x\in K \colon \textnormal{greedy $L$-partition of $x$ exists up until at most block $k$}\}.\]
A sequence is not quasi-normal if and only if for all $L \geq N$ either the greedy $L$-partition $(r_k)_{k\in\mathbb{N}}$ exists and produces a sum $\sum_{k=1}^\infty \frac{1}{r_k}$ that is finite or if the greedy $L$-partition does not exist. Therefore, we may write
\begin{equation}\label{eq:decomposition_N_complement}
  K\setminus\mathcal{N} = \bigcap_{\substack{L\in\mathbb{N}\\L\geq N}}\left(\bigcup_{M\in\mathbb{N}} \mathcal{A}_{L,M} \cup \bigcup_{k\in\mathbb{N}} \mathcal{B}_{L,k}\right).
\end{equation}
We observe that $K\setminus\mathcal{N}$ is contained in a $\sigma$-$\phi$-porous subset if already the set
\[\bigcup_{M\in\mathbb{N}} \mathcal{A}_{N,M} \cup \bigcup_{k\in\mathbb{N}} \mathcal{B}_{N,k}\]
is $\sigma$-$\phi$-porous or rather if $\mathcal{A}_{M}\coloneqq \mathcal{A}_{N,M}$ and $\mathcal{B}_{k}\coloneqq \mathcal{B}_{N,k}$ are $\phi$-porous for all $M,k\in\mathbb{N}$.\\

Let us choose $\phi\colon (0,1) \to (0,1)$ implicitly by defining its inverse
\[\phi^{-1}(t) = t^{\frac{1}{t}}.\]
Writing $\phi$ explicitly is unnecessarily difficult, which is why we choose to only give $\phi$ implicitly.
Note that $\phi^{-1}$ is continuous and strictly increasing on $(0,1)$ since for every $t\in(0,1)$ we have that
\[\mathrm{D}\phi^{-1}(t) = - t^{\frac{1}{t}-2}(\log(t)-1) >0.\]
Also we note that $\lim_{t\to 0} \phi^{-1}(t) = 0$ and $\phi^{-1}(1) = 1$. We conclude that $\phi^{-1}$ is bijective on $(0,1)$, and hence, that $\phi^{-1}$ indeed implicitly defines a function $\phi\colon (0,1) \to (0,1)$.
Since for $t<\frac{1}{3}$ we have
\[\mathrm{D}^2 \phi^{-1}(t) = t^{\frac{1}{t}-4} \left(-3 t+\log ^2(t)+2 (t-1) \log (t)+1\right)>0,\]
we know that $\phi^{-1}$ is convex on $(0,\frac{1}{3})$, and hence, that $\phi$ is concave. Furthermore, since $\phi^{-1}$ is convex and since $\phi^{-1}(t)<t$ for $t<1$ we conclude that also $\phi$ is strictly increasing. Lastly, we note that since $\lim_{t\to 0} \phi^{-1}(t) = 0$ we have $\lim_{s\to0} \phi(s) = 0$. These observations show that $\phi$ is a function suitable for $\phi$-porosity.
\begin{lemma}\label{lem:AM_phi_porous}
  For every $M\in\mathbb{N}$ the set $\mathcal{A}_{M}$ is $\phi$-porous.
\end{lemma}
\begin{proof}
  Let $x\in \mathcal{A}_{M}$ and let us choose
  \[\varepsilon_0 \coloneqq \min\left\{\frac{1}{8M}, 2^{-3N-1}\right\},\quad \alpha = 1,\quad \text{and}\quad \varepsilon\in(0,\varepsilon_0).\]
  Choose $j = \lceil -\log_2(\varepsilon)\rceil$ and note that then $2^{-j}<\varepsilon\leq 2^{-j+1}$. We  now define $y\in K$ by
\[y\coloneqq (x_1, x_2,\dots,x_j,\underbrace{1,\dots,1}_N,1,2,\dots,N,1,2,\dots,N,\dots).\]
  By \eqref{smallDistanceImpliesSameEntriesAtBeginning} we have $d(x,y)<\varepsilon$. Observe that the greedy $N$-partition $(r^y_k)_{k\in\mathbb{N}}$ of $y$ exists and that $\sum_{k\in\mathbb{N}}\frac{1}{r^y_k} = \infty$, since somewhere after $x_j$ all blocks are densely packed without spaces in between. Hence, $y\notin \mathcal{A}_{M}$. Now set
  \[m \coloneqq \left\lfloor - \log_2(\varepsilon) \frac{1}{\varepsilon}\right\rfloor. \]
  Note that we have that
  \[
    \phi^{-1}(\alpha \varepsilon) = \varepsilon^{\frac{1}{\varepsilon}} = 2^{\log_2(\varepsilon)\frac{1}{\varepsilon}} \leq 2^{-\left\lfloor-\log_2(\varepsilon)\frac{1}{\varepsilon}\right\rfloor} = 2^{-m},
  \]
  and therefore, that $B(y,\phi^{-1}(\alpha \varepsilon)) \subseteq B(y,2^{-m})$.
  Now assume that $z\in B(y,\phi^{-1}(\alpha \varepsilon))$. If the greedy $N$-partition of $z$ does not exist we immediately conclude that $z\notin \mathcal{A}_M$. In the following we assume that the greedy $N$-partition $(r^z_k)_{k\in\mathbb{N}}$ of $z$ does exist. Since
  \begin{align*}
    m &\geq -\log_2(\varepsilon) \frac{1}{2\varepsilon} \geq -\log_2(\varepsilon)\frac{1}{2^{-3N}}\geq \lceil-\log_2(\varepsilon)\rceil 2^{3N-1}\\
      &>\lceil-\log_2(\varepsilon)\rceil (1+2N) \geq \lceil-\log_2(\varepsilon)\rceil + 2N = j+2N.
  \end{align*}
  and due to \eqref{smallDistanceImpliesSameEntriesAtBeginning} we have that the leading entries of $z$ are
  \[(z_i)_{i=1}^{m} = (\underbrace{x_1,\dots,x_j,\underbrace{1,\dots,1}_{N},\underbrace{1,2,\dots,N, 1,2,\dots,N}_{\text{at least one full block}},\dots}_{m \text{ entries}}).\]
  Next we show that $\sum_{k\in\mathbb{N}} \frac{1}{r^z_k}$ is too large for $z$ to belong to $\mathcal{A}_M$ by estimating a lower bound on it using only the $r^z_k$ corresponding to blocks that fit entirely in the section of consecutive blocks $1,2,\dots,N$. By construction, we know that a block must begin at index $j+N+1$ and all other such blocks follow without spaces immediately after. For the sum we have
  \begin{align*}
    \sum_{k\in\mathbb{N}}\frac{1}{r^z_k} &\geq \sum_{\substack{k\in\mathbb{N} \\ j+N+1 \leq r^z_k \\ r^z_k + N - 1 \leq m}} \frac{1}{r^z_k}
                                         = \sum_{k = 0}^{\left\lfloor\frac{m-(j+N)}{N}\right\rfloor - 1} \frac{1}{j+N+1+kN}
                                         =\sum_{k = 0}^{\left\lfloor\frac{m-(j+N)}{N}\right\rfloor - 1} \frac{1}{j+1+(k+1)N}\\
                                         &>\int_{t = 0}^{\left\lfloor\frac{m-(j+N)}{N}\right\rfloor - 1} \frac{1}{j+1+(t+1)N}
                                         = \frac{1}{N}\log\left(1+\frac{\left(\left\lfloor\frac{m-(j+N)}{N}\right\rfloor - 1\right)N}{1+j+N}\right)\\
                                         &\geq \frac{1}{N}\log\left(1+\frac{\left(\frac{m-(j+N)}{N} - 2\right)N}{1+j+N}\right)
                                         = \frac{1}{N}\log\left(1+\frac{m-j-3N}{1+j+N}\right)\\
                                         &\geq \frac{1}{N}\log\left(1+\frac{m-j-3N}{1+j+3N}\right)
                                         = \frac{1}{N}\log\left(\frac{m}{1+j+3N}\right) \\
                                         &\geq \frac{1}{N}\log\left(\frac{m}{2j}\right)
                                         =\frac{1}{N}\log\left(\frac{\left\lfloor -\log_2(\varepsilon) \frac{1}{\varepsilon}\right\rfloor}{2\lceil-\log_2(\varepsilon)\rceil}\right)
                                         \geq\frac{1}{N}\log\left(\frac{ -\log_2(\varepsilon) \frac{1}{2\varepsilon}}{4(-\log_2(\varepsilon))}\right)\\
                                         &=\frac{1}{N}\log\left(\frac{1}{8\varepsilon}\right) > \frac{1}{N}\log(M),
  \end{align*}
  which implies that $z\notin \mathcal{A}_M$.
  We conclude that $\mathcal{A}_{M}$ is $\phi$-porous.
\end{proof}
\begin{lemma}\label{lem:Bk_porous}
  For every $k\in\mathbb{N}$ the set $\mathcal{B}_{k}$ is porous.
\end{lemma}
\begin{proof}
  Let $x\in \mathcal{B}_{k}$
  \[\varepsilon_0 = 1, \quad  \alpha = 2^{-N-(k+1)N},\quad \text{and} \quad \varepsilon\in(0,\varepsilon_0).\]
  Choose $j = \lceil -\log_2(\varepsilon)\rceil$. We  now define $y\in K$ by
\[
y\coloneqq (x_1,\dots,x_j,\underbrace{1,\dots,1}_N,1,2,\dots,N,1,2,\dots,N,\dots).
\]
  Note that the greedy $N$-partition $Hindawi Publishing Corporation(r^y_k)_{k\in\mathbb{N}}$ of $y$ exists, and hence, $y\notin \mathcal{B}_{L,k}$. Then for every $z\in B(y, \alpha \varepsilon)\subseteq B(y, 2^{-j-N-(k+1)N})$ we have that $z$ has the leading entries
\[
  (z_i)_{i=1}^{j+N+(k+1)N} = (x_1,\dots,x_j,\underbrace{1,\dots,1}_{N},\underbrace{\overbrace{1,2,\dots,N},\dots,\overbrace{1,2,\dots, N}}_{k+1\textnormal{ many blocks }}).
\]
No block in the greedy $N$-partition of $z$ can begin in the section of consecutive ones, since then $N$ cannot occur in such a block. It follows that the blocks in the greedy $N$-partition after the section of consecutive ones must be chosen in the way they are marked above. Therefore, the greedy $N$-partition of $z$ exists at least up to block $k+1$, and hence, $z\notin \mathcal{B}_{k}$. We conclude that the set $\mathcal{B}_k$ is porous.
\end{proof}
\begin{remark}
  A similar proof can be carried out for all sets $\mathcal{A}_{L,M}$ and $\mathcal{B}_{L,k}$, but as mentioned above, we only need the proved case.
\end{remark}
\begin{theorem}\label{thm:quasi_normals_co_sigma_phi_porous}
  The set $\mathcal{N}$ contains a co-$\sigma$-$\phi$-porous subset.
\end{theorem}
\begin{proof}
  By Equation \eqref{eq:decomposition_N_complement} we see that $K\setminus\mathcal{N}$ is contained in a  subset of all $\mathcal{A}_{M}$ and $\mathcal{B}_{k}$ for $M,k\in\mathbb{N}$, which are $\phi$-porous and porous, respectively. This is the case by Lemma~\ref{lem:AM_phi_porous} and Lemma~\ref{lem:Bk_porous}.
\end{proof}
\begin{corollary}
  The set of quasi-normal sequences $\mathcal{N}$ contains a dense $G_\delta$ subset.
\end{corollary}
\section{Some Extensions}\label{sec:ext}
In this section we take a closer look at why the construction of the set $\mathcal{N}_0$ gives us a meager set.
\subsection{A more General Subset of Quasinormal Sequences}\label{sec:Nf}
   Let us first elaborate why the sequences from Proposition \ref{prop:quasi_normal} were chosen, i.e., we look at the proof of the proposition as given in \cite{BCM2022}.
   \begin{proof}[Proof of Proposistion \ref{prop:quasi_normal}]
      Let $x\in K$ be such that it satisfies the conditions of the proposition.
      We denote by $x_{r_k}$ the first element of the block $\mathcal{R}_k$. Observe that
      \[r_k = 1 + (k-1)L + \sum_{i=1}^{k} |\mathcal{S}_i|.\]
      Now consider a function $f\colon \mathbb{N} \to (0,\infty)$, where $f(r_k) = \log k$. We have by imposing the bound $\sum_{i = 1}^k |\mathcal{S}_i|\leq c k$ for some $c>0$ that
      \begin{align*}
        \sum_{k = 1}^{\infty} \frac{1}{r_k f(r_k)} &= \sum_{k=1}^{\infty} \frac{1}{\left(1 + (k-1)L + \sum_{i = 1}^k |\mathcal{S}_i|\right)\cdot \log k}\\
                                                   &\geq \sum_{k = 1}^{\infty} \frac{1}{(1 + (k-1)L + ck)\cdot \log k}\\
                                                   &\geq \sum_{k = 1}^{\infty} \frac{1}{(L+c)k\cdot \log k}\\
                                                   &=\infty,
      \end{align*}
      which implies that $x$ is quasi-normal.
   \end{proof}
   \begin{remark}
    The proof essentially uses two features: First, it chooses a single function for $f$ and also this function depends on $x$ and the chosen partition in the way that $f$ only depends on $k$. Second, a bound on $\sum_{i = 1}^k |\mathcal{S}_i|$ is imposed to ensure the divergence of the entire series. The following discussion will show, that the second feature, i.e., the bound cannot be chosen such that the resulting subset is large in a topological sense.
   \end{remark}
   In the following discussion we are going to show that no matter what bound on $\sum_{i = 1}^k |\mathcal{S}_i|$ is chosen, we still have that the set of sequences satisfying this condition is small.\\

   Let $(p_c:\mathbb{N}\to (0,\infty))_{c\in C}$ be a countable family of increasing functions. We use the family $(p_c)_{c\in C}$ as a replacement for the linear bound $\sum_{i = 1}^{k} |\mathcal{S}_i|\leq c k$, i.e., now we prescribe
   \[\sum_{i = 1}^{k} |\mathcal{S}_i|\leq p_c(k)\]
   for some $c\in C$. In order for this new bound to also produce quasi-normal sequences, there must for every $x\in K$ exist some $f\colon \mathbb{N} \to (0,\infty)$ with $\lim_{t\to \infty} f(t) = \infty$ such that
   \[\sum_{k = 1}^{\infty} \frac{1}{(1 + (k-1)L + p_c(k))\cdot f(r_k)} = \infty,\]
   see the proof of Proposition \ref{prop:quasi_normal}.
   In the following discussion we adapt the strategy from \cite{thimm2023OnAMeager}.
\begin{definition}
  Let $L\in\mathbb{N}$ and $c\in C$. We then define
\begin{align*}
  P_{L,c}(x) \ratio\leftrightarrow &\exists\textnormal{ a representation } \mathcal{S}_1\mathcal{R}_1\mathcal{S}_2\mathcal{R}_2\dots=x \textnormal{ such that }\\
                              &
  \begin{lAnd}
  \forall k\in\mathbb{N}\colon |\mathcal{R}_k| = L\\
  \forall k\in\mathbb{N}\colon \mathcal{R}_k\textnormal{ contains all elements of  } I\\
  \forall k\in \mathbb{N}\colon \sum_{i=1}^{k} |\mathcal{S}_i(x)|\leq p_c(k)
  \end{lAnd}
\end{align*}
and
\[P(x)\ratio\leftrightarrow \exists L\in\mathbb{N}\colon\exists c>0\colon P_{L,c}(x).\]
\end{definition}
Let $\mathcal{N}_0$ denote the set of all sequences $x\in K$ satisfying the assertion of Proposition~\ref{prop:quasi_normal}, that is,
\[\mathcal{N}_0 \coloneqq \{x\in K\colon P(x)\}.\]
.
Let us for $L\in\mathbb{N}$ and $c>0$ define the set
\begin{align*}
  \mathcal{N}_{L,c} \coloneqq \{x \in K \colon P_{L,c}(x)\}.
\end{align*}
Similarly to the case of the linear bound let us set
\[\mathcal{N}_0 \coloneqq \bigcup_{L\in\mathbb{N}} \bigcup_{c \in C } \mathcal{N}_{L, c}.\]

\begin{theorem}\label{thm:NckpNowhereDense}
  For every $L\in\mathbb{N}$ and $c\in C$ the set $\mathcal{N}_{L, c}$ is nowhere dense.
\end{theorem}
The proof of this theorem is provided in Section \ref{sec:ProofNckpNowhereDense}.
\begin{corollary}
   The set $\mathcal{N}_0$ is meager.
\end{corollary}
This shows that no matter what bound on $\sum_{i = 1}^{\infty}|\mathcal{S}_i|$ is chosen, and no matter what $f$ is chosen, the strategy of Melo, da Cruz Neto, and de Brito does not work in a topological context.

\subsection{Proof of Theorem \ref{thm:NckpNowhereDense}}\label{sec:ProofNckpNowhereDense}
      \begin{lemma}\label{lem:Too_many_ones_no_longer_N_Lc}
      Let $L\in \mathbb{N}, c>0, x\in \mathcal{N}_{L,c}$, and $n\in\mathbb{N}$. Furthermore, let
      \[y \coloneqq (x_1,\dots,x_n,1,1,\dots)\]
      and
      \[m \coloneqq  p_c\left(L \left\lfloor\frac{n}{L}\right\rfloor+1\right) + 2L + 2.\]
      Then
      \[B(y, 2^{-n-m})\cap\mathcal{N}_{L,c} = \emptyset.\]
     \end{lemma}
     \begin{proof}
      Let us fix some arbitrary $z \in B(y, 2^{-n-m})$ and assume that $z\in\mathcal{N}_{L,c}$. Note that the ball may be written as
      \[B(z, 2^{-n-m}) = \{x_1\}\times\cdots\times\{x_n\}\times\underbrace{\{1\}\times\cdots\times\{1\}}_{m} \times \prod_{k = n+m+1}^{\infty} K.\]
      Since $z\in\mathcal{N}_{L,c}$, we can find a partition $z=\mathcal{S}_1 \mathcal{R}_1 \mathcal{S}_2 \mathcal{R}_2\dots$ satisfying all conditions in the definition of $\mathcal{N}_{L,c}$. Let us fix this partition. Now, let $s_k$ denote the index in $z$ of the leftmost entry of $\mathcal{S}_k$.
      By $i$ we denote the smallest index $j$ such that $s_j\geq n+1$, i.e.,
      \[i \coloneqq \min\{j\in\mathbb{N}\colon s_j\geq n+1\}.\]
      This retrieves the beginning of the first block $\mathcal{S}_i$ that starts in or after the section of $z$ with the consecutive ones. Since $\mathcal{R}_{i-1}$ has to contain all elements of $I$ and has a length of $L$ it can at most extend $L-(N-1)$ places into the section of consecutive ones. Since $m>L-(N-1)$ we conclude that in fact $i$ is the index of the leftmost index of the first $\mathcal{S}_k$ to start within the section of consecutive ones, and not after.\\

      Next, we will find an upper bound on $i$. All $\mathcal{R}_k$ have the same length $L$, but the length of the $\mathcal{S}_k$ is arbitrary. Hence, $i$ will be the largest if $\mathcal{S}_1,\dots,\mathcal{S}_{i-1} = ()$. For this reason we follow that
      \begin{equation}\label{eq:lower_bound_on_i_case_with_p}
        i \leq L \left\lfloor\frac{n}{L}\right\rfloor + 1.
      \end{equation}

      We will now establish a lower bound on $|\mathcal{S}_i|$. The block $\mathcal{S}_i$ ends one place before $\mathcal{R}_i$ begins. Since $\mathcal{R}_i$ is of length $L$ and contains all elements of $I$, it can extend at most $L-(N-1)$ places into the section of consecutive ones from the right. This together with the identical condition at the beginning of the section of $m$ consecutive ones gives us that
      \[|\mathcal{S}_i|\geq m - 2(L-(N-1)).\]
      By the definition of $m$ and by \eqref{eq:lower_bound_on_i_case_with_p} we have
      \begin{align*}
        |\mathcal{S}_i|&\geq m - 2(L-(N-1))\\
                       &= p_c\left(L \left\lfloor\frac{n}{L}\right\rfloor+1\right) + 2L + 2- 2(L-(N-1))\\
                       &= p_c\left(L \left\lfloor\frac{n}{L}\right\rfloor+1\right) + 2N\\
                       &\geq p_c(i) + 2N\\
                       &> p_c(i).
      \end{align*}
      Hence,
      \[\sum_{k = 1}^i |\mathcal{S}_k|\geq |\mathcal{S}_i|> p_c(i)\]
      and, since the partition $\mathcal{S}_1\mathcal{R}_1\mathcal{S}_2\mathcal{R}_2\dots$ was arbitrary, we have $z \notin \mathcal{N}_{L,c}$.
      Since $z \in B(y, 2^{-n-m})$
      was arbitrary, we have shown the assertion.
     \end{proof}
    In the previously considered subset $\mathcal{N}_0$ of quasi-normal functions, the function $f$ was chosen as $f(r_k)=\log(k)$. This means that the function $f$ also indirectly depends on a chosen partition $r$ and the fixed $L$ the partition uses. To accommodate and also include such functions $f$ that somehow are of the same type, but may depend on $r$, we instead of a single function define a family $(f^L_r)$ of functions as more accurately stated in the following lemma. It states that also in the case of using some given $f$ in the sum, the greedy $L$-partition is greedy.
\begin{lemma}\label{lem:greedy_L_partition_is_greedy_family_of_fLr}
  Let $L\in\mathbb{N}$ with $L\geq N$ be given and let
  \[\mathcal{P}_L\coloneqq\{(r_k)_{k\in\mathbb{N}}\in\mathbb{N}^\mathbb{N}\colon \forall k\in\mathbb{N}\colon r_{k+1}-r_k \geq L \}\]
  denote the set of possible $L$-partitions. Furthermore, let
  \[f=(f^L_r\colon \mathbb{N} \to (0,\infty))_{\substack{N\leq L\in\mathbb{N}\\r\in\mathcal{P}_L}}\]
  be a parametrized family of functions that for all $L\in\mathbb{N}$ with $L\geq N$ fulfills the following two properties:
  \begin{enumerate}[label=(\roman*)]
    \item $\forall r\in\mathcal{P}_L\colon \displaystyle\lim_{t\to\infty} f^L_r(t) = \infty$
    \item \label{item:f_increasing} $\forall r,\tilde{r}\in\mathcal{P}_L\colon \forall k\in\mathbb{N}\colon (r_k\leq\tilde{r}_k \implies f^L_{r}(r_k) \leq f^L_{\tilde{r}}(\tilde{r}_k)) $
  \end{enumerate}
  Now, let for $x\in K$ and $r \coloneqq (r_k)_{k\in\mathbb{N}}$ denote the indices of the greedy $L$-partition of $x$, and $\tilde{r}\coloneqq (\tilde{r}_k)_{k\in\mathbb{N}}$ a partition of $x$ different from the former with disjoint blocks of length $L$ in which all elements of $I$ occur. Then we have
  \[\sum_{k\in\mathbb{N}} \frac{1}{r_k f^L_r(r_k)} \geq \sum_{k\in\mathbb{N}}\frac{1}{\tilde{r}_k f^L_{\tilde{r}}(\tilde{r}_k)}.\]
  If the right hand side converges we even get a strict inequality.
\end{lemma}
\begin{proof}
  The proof of the statement is identical to the one of Proposition \ref{prop:greedy_L_partition_is_greedy}. We just use the additional fact that $f$ is increasing in the sense of \ref{item:f_increasing} and that $t\mapsto \frac{1}{t}$ is monotonically decreasing for $t>0$.
\end{proof}
\begin{remark}
  The second property in the previous lemma is inspired by the original case where $f^L_r(r_k) = \log k$, and therefore, also fulfills $f^L_{r}(r_k)\leq f^L_{\tilde{r}}(\tilde{r}_k)$ for $k \in \mathbb{N}$.
\end{remark}
We are going to show, that a similar construction to the one of $\mathcal{N}_0$, but without using the bound $\sum_{i = 1}^{k} |\mathcal{S}_i|\leq p_c(k)$, leads to a topologically large set. In essence, this tells us that the part where the attempt of showing that $\mathcal{N}_0$ is topologically large breaks at the point of prescribing the bound $p_c(k)$ and not due to prescribing only a certain type of function $f$.\\

We therefore choose some family $f = (f^L_r\colon \mathbb{N}\to(0,\infty))_{\substack{N\leq L\in\mathbb{N}\\r\in \mathcal{P}}}$ with the following three properties:
\begin{enumerate}[label={(F\arabic*)}]
    \item $\forall r\in\mathcal{P}_L\colon \displaystyle\lim_{t\to\infty} f^L_r(t) = \infty$
    \item $\forall r,\tilde{r}\in\mathcal{P}_L\colon \forall k\in\mathbb{N}\colon (r_k\leq\tilde{r}_k \implies f^L_{r}(r_k) \leq f^L_{\tilde{r}}(\tilde{r}_k)) $
    \item $\forall r \in \mathcal{P}_N\colon \forall k\in\mathbb{N}\colon \sum_{i \in \mathbb{N}} f^L_r(k+iN) = \infty$
  \end{enumerate}
Then we define the property
\begin{align*}
  P_{L,f}(x) \ratio\leftrightarrow &\exists\textnormal{ a representation } \mathcal{S}_1\mathcal{R}_1\mathcal{S}_2\mathcal{R}_2\dots=x \textnormal{ such that }\\
                              &
  \begin{lAnd}
  \forall k\in\mathbb{N}\colon |\mathcal{R}_k| = L\\
  \forall k\in\mathbb{N}\colon \mathcal{R}_k\textnormal{ contains all elements of  } I\\
  \forall k\in \mathbb{N}\colon \sum_{k\in\mathbb{N}} \frac{1}{r_k f^L_r(r_k)} = \infty
  \end{lAnd},
\end{align*}
where $r$ denotes the greedy $L$-partition of $x$.
We then define the set
\[\mathcal{N}_f \coloneqq \{x\in K\colon \exists L\in \mathbb{N}\colon L\geq N \wedge P_{L,f}(x)\}.\]

We are going to show that for every $f$ satisfying the conditions of Lemma \ref{lem:greedy_L_partition_is_greedy_family_of_fLr}, the complement $K\setminus\mathcal{N}_f$ is contained in a  meager subset. For this let us for $L\in\mathbb{N}$ with $L\geq N$ and $M\in\mathbb{N}$ define the set
\[\mathcal{A}_{L,M}\coloneqq \left\{x\in K\colon \textnormal{greedy $L$-partition $(r_k)_{k\in\mathbb{N}}$ of $x$ exists and} \sum_{k\in\mathbb{N}} \frac{1}{r_k f^L_r(r_k)} < M \right\}\]
and for $k\in\mathbb{N}$ the set
\[\mathcal{B}_{L,k} \coloneqq \{x\in K \colon \textnormal{greedy $L$-partition exists up until at most block $k$}\}.\]
Technically, both sets $\mathcal{A}_{L,M}$ and $\mathcal{B}_{L,k}$ should additionally carry the index $f$, but for the sake of better readability we drop it here.\\
A sequence is not in $\mathcal{N}_{f^L_{\cdot}}$ if and only if for all $L \geq N$ either the greedy $L$-partition $(r_k)_{k\in\mathbb{N}}$ exists and produces a sum $\sum_{k=1}^\infty \frac{1}{r_k f^L_r(r_k)}$ that is finite or if the greedy $L$-partition does not exist. Therefore, we may write
\begin{equation}\label{eq:decomposition_N_complement_copy}
  K\setminus\mathcal{N}_{f} = \bigcap_{\substack{L\in\mathbb{N}\\L\geq N}}\left(\bigcup_{M\in\mathbb{N}} \mathcal{A}_{L,M} \cup \bigcup_{k\in\mathbb{N}} \mathcal{B}_{L,k}\right).
\end{equation}
We observe that $K\setminus\mathcal{N}_{f}$ is contained in a meager subset if already the set
\[\bigcup_{M\in\mathbb{N}} \mathcal{A}_{N,M} \cup \bigcup_{k\in\mathbb{N}} \mathcal{B}_{N,k}\]
is meager or rather if $\mathcal{A}_{N,M}$ and $\mathcal{B}_{N,k}$ are nowhere dense for all $M,k\in\mathbb{N}$.
\begin{lemma}\label{lem:AM_nowhere_dense}
  For every $M\in\mathbb{N}$ the set $\mathcal{A}_{N,M}$ is nowhere dense.
\end{lemma}
\begin{proof}
  Let $x\in \mathcal{A}_{N,M}$ and $\varepsilon > 0$. Choose $j\in\mathbb{N}$ such that $2^{-j}<\varepsilon\leq 2^{-j+1}$. We  now define $y\in K$ by
  \[y\coloneqq (x_1, x_2,\dots,x_j,1,2,\dots,N,1,2,\dots,N,\dots).\]
  By \eqref{smallDistanceImpliesSameEntriesAtBeginning} we have $d(x,y)<\varepsilon$. Observe that the greedy $N$-partition $(r^y_k)_{k\in\mathbb{N}}$ of $y$ exists and that $\sum_{k\in\mathbb{N}}\frac{1}{r^y_k f^L_{r^y}(r^y_k)} = \infty$, since somewhere after $x_j$ all blocks are densely packed without spaces in between. Hence, $y\notin \mathcal{A}_{N,M}$. Now set $m\in\mathbb{N}$ such that
  \begin{equation}\label{eq:ReciprocalSumLargerM}
    \sum_{\substack{k \in \mathbb{N}\\r^y_{k}+N-1 < m}} \frac{1}{r^y_k f^L_{r^y}(r^y_k)} > M
  \end{equation}
  and let $n\in\mathbb{N}$ denote the last index $k$ used in the sum, i.e.,
  \[n = \max\{k\in\mathbb{N}\colon r^y_k + N-1 < m\}.\]
  Then by \eqref{smallDistanceImpliesSameEntriesAtBeginning} we have for every $z\in B(y, 2^{-m})$ that
  \[z_1 = y_1 \wedge \dots \wedge z_{r^y_{n}+N-1}=y_{r^y_{n}+N-1}.\]
  If the greedy $N$-partition of $z$ does not exist, we have that $z\notin \mathcal{A}_{N,M}$. In the case that the greedy $N$-partition $(r^z_k)_{k\in\mathbb{N}}$ of $z$ does exist, we also have $z\notin\mathcal{A}_{N,M}$ since $z$ and $y$ match up until index $r^y_n+N-1$, and hence, the first $n$ blocks in the greedy $L$-partition of $z$ and $y$ agree. Therefore, by \eqref{eq:ReciprocalSumLargerM} we have
  \[\sum_{k\in\mathbb{N}}\frac{1}{r^z_k f^L_{r^z}(r^z_k)} \geq \sum_{k = 1}^n \frac{1}{r^y_k f^L_{r^y}(r^y_k)} > M.\]
  We conclude that $\mathcal{A}_{N,M}$ is nowhere dense.
\end{proof}
\begin{lemma}\label{lem:Bk_nowhere_dense}
  For every $k\in\mathbb{N}$ the set $\mathcal{B}_{N,k}$ is nowhere dense.
\end{lemma}
\begin{proof}
  Let $x\in \mathcal{B}_{N,k}$ and $\varepsilon > 0$. Choose $j\in\mathbb{N}$ such that $2^{-j}<\varepsilon\leq 2^{-j+1}$.   We  now define $y\in K$ by
\[
y\coloneqq (x_1,\dots,x_j,\underbrace{1,\dots,1}_N,1,2,\dots,N,1,2,\dots,N,\dots).
\]
  Note that the greedy $N$-partition $(r^y_k)_{k\in\mathbb{N}}$ of $y$ exists, and hence, $y\notin \mathcal{B}_{N,k}$. Then for every $z\in B(y, 2^{-j-N-(k+1)N})$ we have that $z$ has the leading entries
\[
  (z_i)_{i=1}^{j+N+(k+1)N} = (x_1,\dots,x_j,\underbrace{1,\dots,1}_{N},\underbrace{\overbrace{1,2,\dots,N},\dots,\overbrace{1,2,\dots, N}}_{k+1\textnormal{ many blocks }}).
\]
  No block in the greedy $N$-partition of $z$ can begin in the section of consecutive ones, since then $N$ cannot occur in such a block. It follows that the blocks in the greedy $N$-partition after the section of consecutive ones must be chosen in the way they are marked above. Therefore, the greedy $N$-partition of $z$ exists at least up to block $k+1$, and hence, $z\notin \mathcal{B}_{N,k}$
\end{proof}
\begin{remark}
  A similar proof can be carried out for all sets $\mathcal{A}_{L,M}$ and $\mathcal{B}_{L,k}$, but as mentioned above, we only need the proved case.
\end{remark}
\begin{theorem}
  The set $\mathcal{N}_{f}$ contains a co-meager subset of $K$.
\end{theorem}
\begin{proof}
  By Equation \eqref{eq:decomposition_N_complement} we have that $K\setminus\mathcal{N}_{f}$ is contained in a meager subset if all $\mathcal{A}_{M}$ and $\mathcal{B}_{N,k}$ for $M,k\in\mathbb{N}$ are nowhere dense. This is the case by Lemma~\ref{lem:AM_nowhere_dense} and Lemma~\ref{lem:Bk_nowhere_dense}.
\end{proof}
\begin{corollary}
  The set of quasi-normal sequences $\mathcal{N}_{f}$ contains a dense $G_\delta$ subset.
\end{corollary}
 \section{On Sequences with Finite Index Occurrence}\label{sec:finite_index_occurrence}
For the method of alternating projections to work, one needs to use sequences $x\in K$ where every index in $I$ appears infinitely often. Since the previous sections show that almost all sequences are quasi-normal, it is clear that almost all sequences have to contain every index infinitely often. However, at least with the previously given proofs, it seems that there are qualitatively more sequences with infinite index appearance than there are sequences that lead to convergence. This may be suspected as the following theorem shows that the sequences with infinite index occurrence form a co-$\sigma$-porous subset, but the quasi-normal sequences only have been proven to form a co-$\sigma$-$\phi$-porous subset. This tells us that we qualitatively lose some sequences, as $\phi$-porosity is a weaker notion of porosity. If indeed the quasi-normal sequences are non-co-$\sigma$-porous this would stand in contrast to the measure theoretic case, where the quasi-normal sequences and the ones with infinite index occurrence both have full measure, and qualitatively no sequences are lost when making this transition. Investigating if the quasi-normals are indeed non-co-$\sigma$-porous might be an interesting extension of the presented results.
\begin{theorem}
The set
\[\mathcal{F}\coloneqq \{x\in K\colon \exists n\in I\colon\exists M\in\mathbb{N} \colon |\{k\in\mathbb{N}\colon x_k = n\}|<M\}\]
of all sequences in $K$ in which some index $n\in I$ appears only finitely often is co-$\sigma$-porous.
\end{theorem}
\begin{proof}
By defining
\[\mathcal{F}_{n,M} \coloneqq \{x\in K \colon |\{k\in\mathbb{N}\colon x_k = n\}| < M\}\]
we may write
\[\mathcal{F} = \bigcup_{n\in\mathbb{N}} \bigcup_{M\in \mathbb{N}} \mathcal{F}_{n,M}.\]
We show that for all $n,M\in\mathbb{N}$ the set $\mathcal{F}_{n,M}$ is porous. For this let $x\in\mathcal{F}_{n,M}$, let
\[\varepsilon_0 \coloneqq 1, \quad \alpha = 2^{-M-1}, \quad \text{and} \quad \varepsilon\in(0,\varepsilon_0)\]
and choose $j\in\mathbb{N}$ such that $2^{-j}<\varepsilon\leq 2^{-j+1}$. We now define
\[y = (x_1,\dots,x_j,n,n,n,\dots).\]
By \eqref{smallDistanceImpliesSameEntriesAtBeginning} we have that $d(x,y)<\varepsilon$. By the same relation we observe that all $z\in B(y,\alpha \varepsilon)\subseteq B(y,2^{-j-M-1})$ have at least $M$ many occurrences of $n$. Hence, $B(y,\alpha \varepsilon)\cap\mathcal{F}_{n,M} = \emptyset$, and therefore, $\mathcal{F}_{n,M}$ is porous. This directly implies the desired result, as $\mathcal{F}$ may be written as a countable union of porous sets.
\end{proof}

\noindent\textbf{Acknowledgements: }
Special thanks go to my supervisor Eva Kopeck\'{a} for many valuable discussions and suggestions.

\printbibliography

\end{document}